\documentclass[10pt]{article}
\usepackage[margin=3.5cm]{geometry}

\usepackage{amsmath}
\usepackage{amssymb}
\usepackage{amsthm}
\usepackage{diagbox}
\usepackage{pgf}

\newcommand{\A}{\ensuremath{\operatorname{Aut}}}
\newcommand{\D}{\ensuremath{\operatorname{D}}}
\newcommand{\cf}{\mathfrak{c}}

\newtheorem{theorem}{Theorem}
\newtheorem{lemma}{Lemma}
\newtheorem{proposition}[theorem]{Proposition}
\newtheorem*{maintheorem}{Main Theorem}

\theoremstyle{definition}
\newtheorem{definition}{Definition}
\newtheorem*{acknowledgement}{Acknowledgement}

\theoremstyle{remark}
\newtheorem{example}[theorem]{Example}

\renewenvironment{proof}{\par\medskip\noindent{\bf Proof\ \ }}{\qed\medskip}

\begin{document}

\title{
	Distinguishing locally finite trees
}

\author{
	Svenja H\"{u}ning 
	\thanks{Technische Universit\"at Graz, 8010 Graz, Austria}
	\and Wilfried  Imrich 
	\thanks{Montanuniversit\"{a}t Leoben, 8700 Leoben, Austria}
	\and Judith Kloas 
	\footnotemark[1]
	\and Hannah Schreiber 
	\footnotemark[1]
	\and Thomas W.~Tucker 
	\thanks{Colgate University, Hamilton NY 13346, USA}
}
\date{\today}
\maketitle

%


\begin{abstract}
	The distinguishing number $\D(G)$ of a graph $G$ is the smallest number of 
	colors that is needed to color  the vertices of $G$ such that the only color 
	preserving automorphism is the identity. For infinite graphs $\D(G)$ is 
	bounded by the supremum of the valences, and for finite graphs by 
	$\Delta(G)+1$, where $\Delta(G)$ is the maximum valence. Given a finite or 
	infinite tree $T$ of bounded finite valence $k$ and an integer $c$, where 
	$2 \leq c \leq k$, we are interested in coloring the vertices of $T$ by $c$ 
	colors, such that every color preserving automorphism fixes as many vertices 
	as possible. In this sense we show that there always exists a $c$-coloring 
	for which all vertices whose distance from the next leaf is at least 
	$\lceil\log_ck\rceil$ are fixed by any color preserving automorphism, and 
	that one can do much better in many cases.
\end{abstract}

\section{Introduction}

This paper is concerned with automorphism breaking of finite and infinite trees 
of bounded valence by vertex colorings. 1977 Babai \cite{ba-77} showed
that the vertices of every $k$-regular tree, where $k \geq 2$ is an arbitrary 
cardinal, can be colored with two colors such that only the identity
automorphism preserves the coloring. For trees that are not regular such a 
$2$-coloring need not be possible. This raises the question of how many
colors are needed to break all automorphisms of a given tree $T$, that is, of 
finding the smallest cardinal $d$ to which there exists a
$d$-coloring of the vertices of $T$ that is only preserved by the identity 
automorphism.

Another question is to find,  for a given $c \geq 2$, $c$-colorings of the  
vertices of a given tree $T$ such that the color-preserving automorphisms fix 
subtrees of $T$ that are maximal in some sense. This is the problem, which we 
consider here.

Our note is related to \cite{imhu-2017}, where both questions were answered for 
connected graphs of maximum valence 3. That paper, in turn, was
motivated by the general problem of determining the distinguishing number of 
graphs and the Infinite Motion Conjecture of Tucker.

The \emph{distinguishing number} $\D(G)$ of a graph $G$ is the smallest cardinal 
number $d$ such that there exists a $d$-coloring of the vertices
of $G$ which is only preserved by the identity automorphism. We also say that 
such a coloring \emph{breaks} $\A(G)$ and that $G$ is 
\emph{$d$-distinguishable}.
These concepts were introduced 1996 by Albertson and Collins \cite{alco-96} and 
have spawned a series of related papers.
In two of them, by Collins and Trenk~\cite{CT06} and Klav\v{z}ar, Wong and 
Zhu~\cite{klwo-2006}, it is shown that the distinguishing number $\D(G)$
of any finite connected graphs $G$ of maximal valence $k$ is at most $k$, unless 
$G$ is $K_k$, $K_{k,k}$ or $C_5$. Then $\D(G) = k + 1$.
In~\cite{imkltr-2007} this was extended to infinite graphs. In that case $\D(G)$ 
is bounded by the supremum of the valences of the vertices.

Despite the fact that the distinguishing number can be arbitrarily large, it 
is~$1$ for asymmetric graphs. As almost all finite graphs are asymmetric, this 
means that almost all graphs have distinguishing number~$1$. Furthermore, almost 
all finite graphs that are not asymmetric have just one non-identity 
automorphisms, which is an 
involution\footnote{Florian Lehner, private communication.}. One can break it by 
coloring one selected vertex black and all others white. Clearly such graphs are 
$2$-distinguishable.

For infinite graphs we have the Infinite Motion Conjecture~\cite{tu-11}. It says 
that all connected, locally finite, infinite graphs are $2$-distinguishable if 
every non-identity automorphism moves infinitely many vertices. Despite the fact 
that it is true for many classes of graphs, for example for graphs of 
subexponential growth, see Lehner~\cite{le-15}, the conjecture is still open. 
Until recently it was not even clear whether it holds for graphs of maximum 
valence~$3$, but in~\cite{imhu-2017} it was shown that all connected infinite 
graphs of maximal valence 3 are 2-distinguishable and that no motion assumption 
is needed.

In the case of connected finite graphs $G$ of maximum valence~$3$ it is even 
enough to require that every automorphism moves at least three vertices to 
ensure $2$-distinguishability, unless $G$ is the cube or the Petersen graph. In 
fact, with the exception of $K_4$, $K_{3,3}$, the cube and the Petersen graph, 
all connected finite graphs $G$ of maximum valence~$3$ admit a $2$-coloring 
where every automorphism that preserves the coloring fixes all vertices, with 
the exception of at most one pair of interchangeable vertices. We say the $2$-
coloring fixes all but this pair of vertices.
For example, consider three copies of $K_{1,3}$, select a vertex of valence~$1$ 
in each copy and identify them. The resulting tree has distinguishing 
number~$3$, and it is easy to find a $2$-coloring that fixes all but two 
vertices.

In this note we generalize this to trees $T$ of maximal valence $k$ and an 
arbitrary number of colors $c$. We wish to find $c$-colorings that 
\emph{fix as many vertices as possible}, that is, we wish to maximize the sets 
of vertices that are fixed by each color preserving automorphism. As this is 
hard to control, we look for the smallest number $r(c,k)$ such that there exists 
a $c$-coloring of $T$ that fixes in the worst case at least all vertices of $T$ 
whose distance from the next leaf is at least 
$r(c,k) = \left\lceil \log_c k \right\rceil$. This is made more precise in 
Section~\ref{sec: main theorem}, where one can see that in a lot of cases this 
number will be significantly smaller.

\section{Preliminaries}

\subsection{Graph representation}

Let $G = (V(G),E(G))$ be a connected graph, $V(G)$ its set of vertices and 
$E(G)$ its set of edges. To simplify the notation we write $V$ and $E$ if the 
graph is clear by the context.
If the vertices of $G$ have maximal valence 3, then $G$ is called 
\emph{subcubic}.
As usual we call the number of edges on a shortest path between two vertices $u$ 
and $v$ be the \emph{distance} $d(u,v)$ between $u$ and $v$.

\begin{definition}
	Let $v \in V$. The \emph{ball} of radius $n$ with center $v$ is defined as 
	the set $B(n,v) = \{u \in V \,|\, d(u, w) \leq n\}$.
	The sphere of radius $n$ around $v$ is the set 
	$S(n,v) = \{u \in V \,|\, d(u, v) = n\}$, that means it is the set of all 
	vertices of distance $n$ from $v$, see Figure~\ref{fig:levels}.
\end{definition}

In this paper, we mostly represent a graph $G$ by an arrangement of spheres with 
a common center $v \in V(G)$, and say that $G$ is \emph{rooted} in $v$. In that 
context, a \emph{down-neighbor}, a \emph{cross-neighbor} and an 
\emph{up-neighbor} of $u \in S(n,v)$ is a neighbor $w$ of $u$ that it is in 
$S(n-1,v)$, $S(n,v)$ or $S(n+1,v)$ respectively; see Figure~\ref{fig:levels}. 
The corresponding edges are called \emph{down-edges}, \emph{cross-edges} and 
\emph{up-edges}.

\begin{definition}
	Two vertices $z,z'$ are \emph{siblings} if they have the same 
	down-neighbor. We call a vertex $w$ an \emph{only child} of a vertex $u$
	if $w$ is the only neighbor of $u$ of valence $1$.
\end{definition}

\begin{figure}[h]
	\centering
	\includegraphics[scale=0.9]{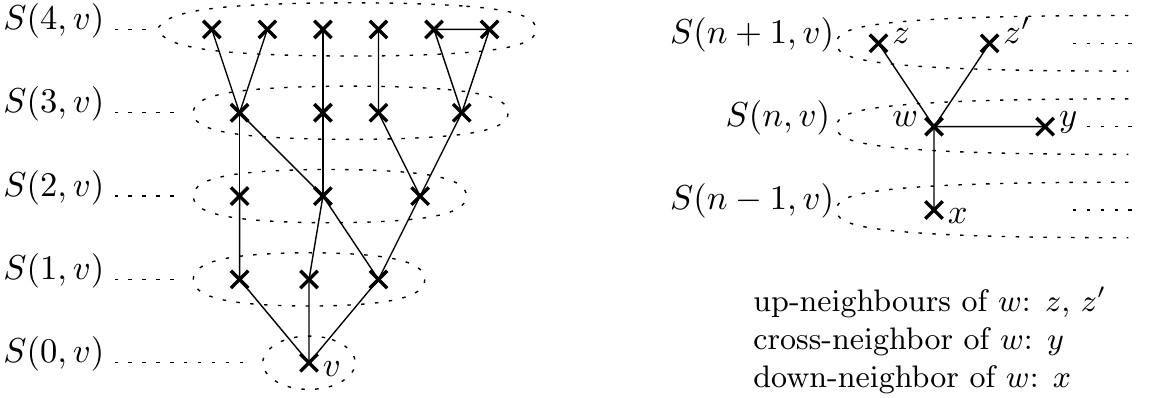}
	\caption{Decomposition of a graph into spheres centered at a vertex $v$. 
	Each vertex in $S(i,v)$ has distance i from $v$.} \label{fig:levels}
\end{figure}

\subsection{Ends and rays}

A \emph{subgraph} of a graph $G = (V(G), E(G))$ is a graph $H = (V(H), E(H))$ 
such that $V(H) \subseteq V(G)$ and $E(H) \subseteq E(G)$.
If $E(H)$ contains all edges between vertices of $V(H)$ that are also in $E(G)$, 
we say $H$ is an \emph{induced subgraph} of $G$ and denote it by 
$\langle V(H)\rangle$.
If $S \subset V(G)$, then $G \setminus S$ is the subgraph of $G$ induced by the 
vertices of the set $V(G) \setminus S$.

\begin{definition}
	A \emph{ray} is an infinite graph $R=(V, E)$, with 
	$V = \{v_0, v_1, v_2, ...\}$ and $E = \{v_0v_1, v_1v_2, v_2v_3, ...\}$
	where the $v_{i}$ are pairwise different.
	If a ray $R_1=(V_1, E_1)$ is a subgraph of a ray $R_2=(V_2, E_2)$, 
	then $R_1$ is called a \emph{tail} of $R_2$.
\end{definition}

Two rays $R_1$ and $R_2$ in a graph $G = (V(G), E(G))$ are \emph{equivalent}, in 
symbols $R_1 \sim R_2$, if for every finite set $S \subset V(G)$ there exists a 
connected component of $G \setminus S$ containing a tail of both 
$R_1$ and $R_2$.
One can show that $\sim$ is an equivalence relation. The equivalence classes of 
$\sim$ are called \emph{ends} of $G$.

If $T = (V,E)$ is an infinite tree, then the set of vertices in $V$ such that 
$T-v_i = \langle V \setminus v_i \rangle$
contains at least two infinite components is denoted by 
$V^C = \{v_1, v_2, \dotsc\}$. We call the induced subgraph $\langle V^C \rangle$ 
the \emph{trunk} of $T$ and denote it  by $T^C$.

\subsection{Automorphisms and Colorings}

An \emph{isomorphism} $\varphi: V(G) \to V(H)$ between two graphs $G$ and $H$ is 
a bijection such that $uv \in E(G)$ if and only if 
$\varphi(u)\varphi(v) \in E(H)$.
The isomorphisms of $G$ onto itself are called \emph{automorphisms}. 
They form a group $A(G)$.
The \emph{stabilizer} of a vertex $v \in V(G)$ is the set 
$A(G)_v = \{\alpha \in A(G) \,|\, \alpha(v) = v\}$.

\begin{definition}
	A \emph{$d$-coloring} $\cf: V \to \{1, ..., d\}$ of a graph $G$ is 
	a map that gives each vertex of $G$ a color $i\in\{1,...,d\}$.
	We say that an automorphism \emph{preserves} a coloring $\cf$, 
	if $\cf(v) = \cf(\varphi(v))$ for all $v\in V$.
	Otherwise, we say that the coloring \emph{breaks} the automorphism.
\end{definition}

The set $A(G)_\cf$ of all automorphisms preserving $\cf$ forms a group.
If $v \in V(G)$ and $A(G)_\cf \subseteq A(G)_v$, 
we say that $\cf$ \emph{fixes} $v$.
As mentioned in the introduction our aim is to construct graph colorings that 
fix as many vertices as possile.
The minimal number of colors needed to fix all vertices is called the 
distinguishing number.

\begin{definition}
	The \emph{distinguishing number} $\D(G)$ of a graph $G$ is the the smallest 
	$d$ for which there exists a $d$-coloring $\cf$ of $G$ such that
	the only automorphism preserving $\cf$ is the identity.
\end{definition}

\begin{definition}
	The \emph{center} of a finite graph is a vertex for which the greatest 
	distance from $v$ to any other vertex of the graph is minimal.
\end{definition}

In a finite tree $T$ the center is either a single vertex or an edge.
If the center of $T$ is a vertex $w$, then $A(T)_w = A(T)$. If the center is an 
edge $uv$, any automorphism $\alpha \in A(T)$ satisfies either
$\alpha(v) = v$ and $\alpha(u) = u$ or $\alpha(v) = u$ and $\alpha(u) = v$.
We define a subtree $T_w$ of the tree $T$ rooted in $v$, with $w \in S(n,v)$, as 
the tree induced by $w$ and all the vertices in $S(m,v)$, $m > n$,
for which there exists a path to $w$ in $T$ not containing $v$.

\section{Coloring locally finite infinite trees}\label{sec:more_colors}

In this section we prove that every locally finite infinite tree with maximal 
valence $k$ has distinguishing number $k-1$.
For the proof we need the following two lemmas.

\begin{lemma}\label{le:fix_general}
	Let $T$ be a  tree with maximal valence $k$. If $T$ has a vertex $v$ of 
	valence $1 \leq val(v) \leq k-1$ then $\D\left(A(T)_v\right) \leq k-1$.
	In other words, there is a $(k-1)$-coloring $\cf$ of $T$ that breaks all 
	automorphisms of $A(T)_v$.
\end{lemma}

\begin{proof}
	Let $v$ be a vertex of valence $1 \leq val(v) \leq k-1$. We color $v$ with 
	an arbitrary color and all its neighbors with different colors.
	This coloring fixes $B(1,v)$ in $A(T)_v$.
	We prove that for all $i>0$ there is a $k-1$-coloring of $B(i,v)$ that 
	breaks all automorphisms of $A(B(i,v))_v$. To show this, it suffices
	to extend a given coloring $\cf(i)$  of $B(i,v)$ with $k-1$ colors to 
	$B(i+1,v)$. We do this as follows. Every vertex in $S(i,v)$ has at most 
	$k-1$ up-neighbors.
	Each of them can be colored with a different color. If we continue to color 
	the tree in this way, we obtain a coloring that breaks
	all automorphisms of $A(T)_v$ but the identity.
\end{proof}

The second result that we need is K\"onigs Lemma, see for example \cite{di-06}.

\begin{lemma}[K\"onig's Lemma]
	Let $V_{0}, V_{1}, V_{2},\dots$ be an infinite sequence of non-empty, 
	disjoint sets. Let $V = \cup_{i\geq 0}$ denote their union.
	If $G = (V,E)$ is a graph such that for all $v \in V_{n}$, $n \geq 1$, there 
	exits a vertex $f(v) \in V_{n-1}$ adjacent to $v$,
	then there exists an infinite path $v_{0}v_{1}\dots$ in $G$ with 
	$v_{n} \in V_{n}$ for all $n \geq 0$.
\end{lemma}

Note that parts of the next result (namely the finite case) can be extracted 
from a recent classification result from \cite{imhu-2017}.
Here we provide a direct and constructive proof which also generalizes to trees 
with maximal valence $k$.

\begin{theorem}\label{thm:inftree}
	Every locally finite, infinite tree with maximal valence $k$ has 
	distinguishing number $D(G) \leq k-1$.
\end{theorem}

\begin{proof}
	Let $T$ be a locally finite infinite tree. It is well known that such a tree 
	without vertices of valence~$1$ is $2$-distinguishable,
	see for example \cite{wazh-2007}.
	Hence, we can assume that $T$ has at least one vertex of valence~$1$.

	By K\"{o}nig's Lemma, $T$ has at least one end. We first treat the case, 
	where $T$ has exactly one end, say $e$.
	Let $R \in e$ be a ray with origin $v_0$, where $v_0$ is a vertex of 
	valence~$1$.

	We color all vertices of $R$ with color~$1$. Thus every automorphism that 
	stabilizes $R$ also fixes all vertices of $R$.
	Consider a vertex $z$ in $R$ and its $j$ neighbors that are not in $R$. 
	Since $j\leq k-2$ we can color these vertices with different colors
	without using color~$1$. We proceed by doing this for all neighbored 
	vertices of $R$. Let $v$ be one neighbor of $z$ not in $R$ and let $T_v$ be 
	the component of $T-z$ that contains $v$. Now, color $T_v$ as in
	Lemma~\ref{le:fix_general}. Continue by coloring the neighbors of all 
	vertices in $R$ in the same way.
	If $w$ is another vertex of $T$ of valence~$1$, then the unique ray $R_w$ in 
	$e$ with origin $w$ contains at least one vertex which is not colored 
	with~$1$. Hence, independently of how the coloring of $T$ will be finished, 
	any color preserving automorphism $\alpha$ of $T$ will satisfy 
	$\alpha(v_0) \neq w$. Because this holds for all vertices of valence~$1$ 
	that are different from $v_0$ we infer that $\alpha(v_0) = v_0$ and that 
	$\alpha$ fixes every vertex of $R$. Because the vertices in $R$ are fixed 
	also the neighbors and finally all finite trees connected to $R$ are fixed.

	We now treat the case when $T$ has two or more ends. Consider the trunk 
	$T^C$ of $T$ which is a locally finite, infinite tree without leaves.
	Since it is unique any automorphism of $T$ leaves $T^C$ invariant. 
	By~\cite{wazh-2007} we know that $T^C$ is $2$-distinguishable.
	Any vertex $z$ of $T^C$ has at most $k - 2$ neighbors not in $T^C$ that can 
	be fixed by coloring them with different colors.
	Again, for each of these neighbors $v_i$ of $z$ we consider the underlying 
	tree $T_{v_i}$ of $T-z$ that contains $v_i$ as before
	and apply Lemma \ref{le:fix_general}.	
\end{proof}

The following example shows that Theorem~\ref{thm:inftree} cannot be generalized 
to finite trees.

\begin{example}\label{exa:finitesubcubic}
	Consider a finite tree with center $v$ where every vertex is of valence~$3$ 
	or~$1$ and every leaf has the same distance from $v$.
	The distinguishing number of such trees is~$3$. For any $2$-coloring of such 
	a tree, there remains a pair of leaves that is indistinguishable,
	see Figure~\ref{fig:coloring_tuck_ex1}. An other easy example is the 
	$K_{1,3}$.
	
	\begin{figure}[h]
		\centering
		\includegraphics[scale=0.9]{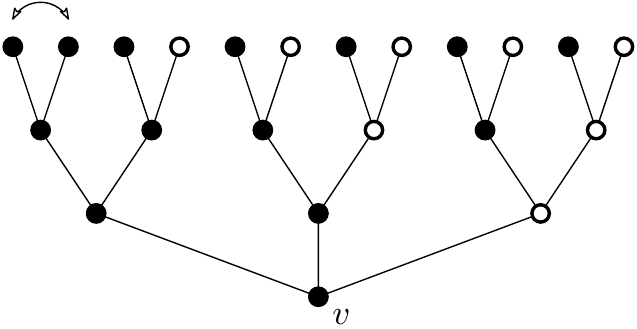}
		\caption{Example of a finite tree $T$ with $\D(T) = 3$. 
		With a $2$-coloring, two leaves are interchangeable.}
		\label{fig:coloring_tuck_ex1}
	\end{figure}
\end{example}

\section{Finite trees} \label{sec:finite_trees}

We begin with the analog of Theorem~\ref{thm:inftree} for finite trees.

\begin{lemma}\label{co:k-1_coloring}
	Let $T$ be a finite tree with maximal valence $k$. Then there is a 
	$k-1$-coloring, which fixes all vertices, with the possible exception 
	of two sibling leaves.
\end{lemma}

\begin{proof}
	If the center of $T$ consists of a single vertex $v$, assign $v$ an 
	arbitrary color and color its neighbors $v_1, v_2,\dotsc $ with as many
	different colors as possible. Only in the case where $v$ has valence $k$ we 
	have to use one color twice. As before let $T_{v_i}$ be the
	component of $T-v$ containing $v_i$. We extend the coloring of $T_{v_i}$ as 
	in Lemma \ref{le:fix_general} for all $i$. In the case where $v$
	has $k$ neighbors it is possible that there is an automorphism that 
	interchanges the two subtrees, say $T_{v_m}$ and $T_{v_n}$
	(with $v_m$ and $v_n$ of the same color). We take an endpoint $a$ of 
	$T_{v_m}$ and change its color so that $T_{v_m}$ and $T_{v_n}$ are
	indistinguishable.

	However, now there might be a color preserving automorphism that moves $a$. 
	This is only possible if there is another vertex of valence~$1$,
	say $b$, that has a common neighbor with $a$. This is the only pair of that 
	kind in $T$.

	We still have to consider the case where the center of $T$ is an edge, 
	say $uv$. In that case we color $u$ and $v$ with different colors,
	and consider the components $T_u$ and $T_v$ of $T-uv$ that contain $u$ or 
	$v$, respectively. We now obtain a distinguishing $k-1$-coloring
	by applying Lemma \ref{le:fix_general} to $T_u$ and $T_v$.
\end{proof}

It is known from Babai \cite{ba-77} that each (infinite) homogeneous tree of 
valence $k > 2$ is $2$-distinguishable. This result cannot be adapted
to the finite case, but note that we can fix all vertices except the leaves by a 
$2$-coloring.

\begin{lemma} \label{lem:regular_tree}
	Let $T$ be a finite tree where every vertex has valence $1$ or $k$. 
	Then there exist a 2-coloring of $T$ that fixes all vertices
	except (some of) the leaves.
\end{lemma}

\begin{proof}
	Suppose the center of $T$ consists of the single vertex $v$. 
	This vertex is then automatically fixed by any automorphism.
	Let $v$ be a white vertex and color all $k$ neighbors of $v$ black.
	
	\begin{figure}[h]
		\centering
		\includegraphics[scale=0.8]{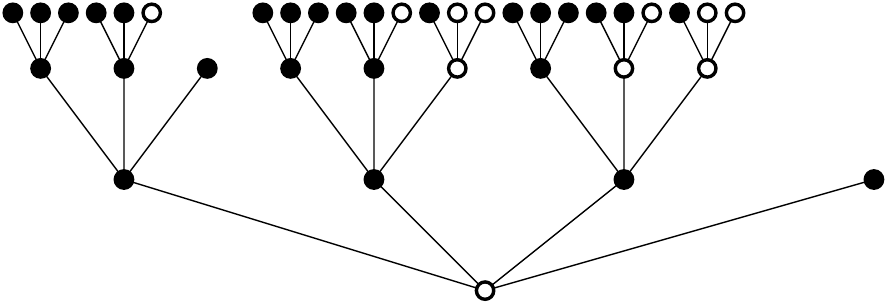}
		\caption{Example of the coloring algorithm given in the proof of 
		Lemma~\ref{lem:regular_tree} with $k = 4$.}
		\label{fig:coloring_ex_lem1}
	\end{figure}

	To avoid that they can be changed we assign different colorings to the next 
	$k-1$ up-neighbors. Since the number of different $2$-colorings
	of $k-1$ indistinguishable vertices is $k$, each of the black vertices can 
	be fixed. We proceed with this coloring process until we reach the leaves.
	Then, all vertices except the leaves are fixed. 
	See Figure~\ref{fig:coloring_ex_lem1} for an example.
	
	If the center of $T$ is an edge, say $uv$, we color $u$ white, $v$ black and 
	continue with a coloring for the subtrees $T_u$ 
	respectively $T_v$ as before.
\end{proof}

\section{Main Theorem}\label{sec: main theorem}

Given a tree $T$ with maximal valence $k$, we now construct a $c$-coloring of 
$T$ which fixes all vertices of $T$ that are sufficiently far away from the next 
leaf, respectively all vertices if there are no leaves.

\begin{maintheorem} \label{new_mainthm}
	Let $T$ be a finite or infinite tree of maximal valence $k < \infty$. 
	Suppose we are given $c$ colors, $2 \leq c \leq k$, to color the vertices of 
	$T$ and that $r$ is an integer that satisfies
	\begin{align*} 
		3 < k \leq 2^{r-1} \mbox{ for } c = 2, \mbox{  or}\\
		3 < k \leq c^r(c-1) + 2 \mbox{ for } c > 2.
	\end{align*} 
	Then there exists a $c$-coloring of $T$ which fixes all vertices of $T$ 
	whose distance from the next leaf is at least $r$.

	Furthermore, $r = 0$ for $c = k$ and $r = 1$ for $c = k - 1$.
\end{maintheorem}

The proof of the main theorem is based on the following proposition, which 
actually contains more information than the main theorem.
In the proposition we let $T$ be a tree rooted in a vertex $v$. If $T$ is 
finite, then  we define $T_u$ as the vertex $u$ together with the components of 
$T-u$ that do not contain $v$. In the infinite case $T_u$ consists of the vertex 
$u$ together with all finite subtrees of $T-u$, 
see Figure~\ref{fig:tu_inifinite}.

\begin{figure}[h]
	\centering
	\includegraphics[scale=0.8]{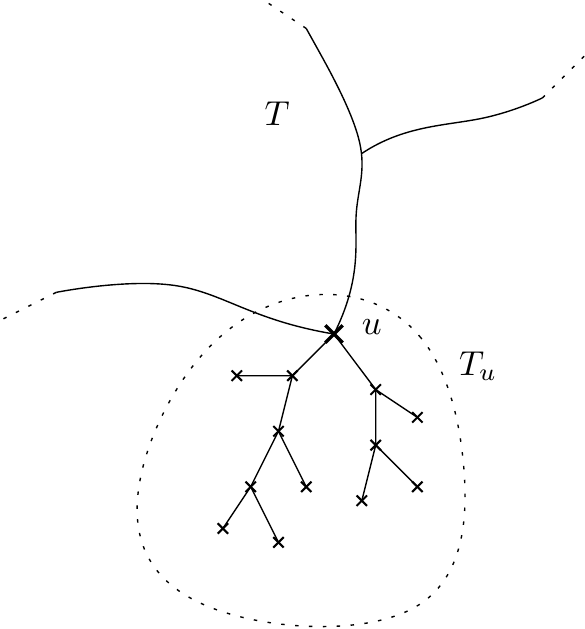}
	\caption{After removal of $u$, all remaining finite trees are attached 
	to $u$ in $T_u$.} \label{fig:tu_inifinite}
\end{figure}

\begin{proposition} \label{thm:fix_tree}
	Let $T$ be a finite or infinite tree of maximal valence $k < \infty$. We 
	assume $T$ to be rooted in a fixed vertex  $v$ (the center of the tree in 
	the finite case) and choose a number $c\geq 2$. Then, for every pair $c,k$, 
	there exists a number $r(c,k)$ and a $c$-coloring of $T$ that fixes all 
	vertices $u \in V(T)$, for which there exists a leaf $w$ in $T_u$ such that
	$d(u,w) \geq r(c,k)$.
	If $k\in \{0,1,2\}$ or $c \geq k$, then $r(c,k) = 0$, which means that the 
	entire graph is fixed. If $k \geq 3$ and $c = k - 1$, then $r(c,k) = 1$. 
	Otherwise
	\begin{align}
		r(c,k) :=
		\begin{cases}
			\log_{2}(\max \left\{3, k-2\right\}) + 1 
				& \text{for } k \geq 4 \text{ and } c = 2.\\
			\log_{c}\left(\max
			\left\{3,\left\lceil \frac{k-2}{c-1}\right\rceil \right\} \right)
				& \text{for } k \geq 4 \text{ and } 2 < c \leq k-2.
		\end{cases}\label{ali:definition_r(k)}
	\end{align}
\end{proposition}

For the values of $\lceil r(c,k) \rceil$ compare Table~\ref{table2}, which is 
based on  Lemma~\ref{co:k-1_coloring}, Proposition~\ref{thm:fix_tree}, and 
summarized in the Main Theorem. In particular, it shows that,  whenever the 
entry in the table is 1, we find a $c$-coloring that fixes all vertices of the 
tree with the possible exception of leaves.

\begin{table}[h]
	\centering
	\begin{tabular}{|l|ccccccccccccccc|}
		\hline
   		\diagbox{\(c\)}{\(k\)}
     	& 2		& 3		& 4		& 5		& 6		& 7		& 8		& 9
     	& 10	& 11	& 12	& 13	& 14	&  15	& 16 \\
  		\hline
  		2
  		& 0		& 1 	& 3 	& 3 	& 3 	& 4 	& 4		& 4
  		& 4		& 5		& 5		& 5		& 5		& 5		& 5 \\
	 	3
	 	& - 	& 0		& 1		& 1		& 1 	& 1 	& 1		& 2
	 	& 2 	& 2		& 2		& 2		& 2		& 2		& 2 \\
	    4
	    & -		& - 	& 0		& 1		& 1		& 1		& 1		& 1
	    & 1		& 1		& 1 	& 1  	& 1    	& 2  	& 2 \\
   		5
   		& - 	& -		& - 	& 0		& 1  	& 1		& 1		& 1
   		& 1 	& 1		& 1	 	& 1 	& 1 	& 1		& 1 \\
  		6
  		& - 	& -		& -		& -		& 0 	& 1 	& 1		& 1
  		& 1 	& 1		& 1		& 1		& 1		& 1 	& 1 \\
   		7
   		& - 	& -		& -		& -		& -		& 0 	& 1 	& 1
   		& 1		& 1		& 1		& 1		& 1		& 1		& 1 \\
  		\hline	
	\end{tabular}
	\caption{Values of $\lceil r(c,k) \rceil$ according to 
	Proposition~\ref{thm:fix_tree}.}
	\label{table2}
\end{table}

The strategy to prove this proposition the following. First we introduce an 
explicit coloring algorithm and then we show two of its properties
(see Lemmas~\ref{lem:combi} and~\ref{lem:max}).

We start by considering a finite tree $T$ with maximal valence 
$4 \leq k < \infty$.
We assume that we have $2 \leq c \leq k-2$ colors. For simplicity we write 
$\lbrace 0,1,2,3,...,c-1\rbrace$ for the \(c\) different colors
(in the following figures, $0$ is white, $1$ black and $2$ gray).

Furthermore, a $c$-coloring of a set of siblings \emph{optimal} if the maximal 
number of vertices with the same color is minimal.
Moreover,  a vertex $u$ is said to satisfy the \textit{distance condition}  
if there exists a leaf $w$ in $T_{u}$ of distance $d(u,w) \geq r(c,k)$.

\subsection{Coloring Algorithm}

We first assume that the center of $T$ consists of a single vertex \(v\).
We root $T$ in $v$ and color $v$ with color $0$. Let $n' \in \mathbb{N}$ be the 
maximal radius such that $S(n',v) \neq \emptyset$.
Now consider the following steps, each of which maybe processed several times..
	
\textit{Step 0:} If there exist indices $\ell$ in $\lbrace 1,\dots , n' \rbrace$ 
such that there still exist uncolored vertices in $S(\ell,v)$,
then let $n$ be the minimum of these indices and continue with Step~$1$.

Otherwise stop the coloring algorithm\footnote{In this step we always look after 
the lowest sphere where there are uncolored vertices, such that in the end we 
can guarantee that we colored all vertices.}.

\textit{Step~1}: If there exists an uncolored vertex in $S(n,v)$ call it $u$ and 
go to Step~2. Otherwise go back to Step~0.

\textit{Step~2:} If the vertex $u$ does not fulfill the distance condition, 
color $u$ with $0$ and go back to Step~1.
Otherwise continue with Step~3.
	
\textit{Step~3:} Note that the vertex $u$ satisfies the distance condition. 
Consider $u$ and all of its uncolored siblings 
\(\lbrace v_{1},\dots v_{r} \rbrace\) which fulfill the distance condition.
If this set is not empty, then color them optimally.
If there are no indistinguishable vertices within 
\(\lbrace u, v_{1},\dots v_{r} \rbrace\) with the same color, go back to Step~1. 
Otherwise continue the coloring in the following way 
(which is still part of Step~3):
	
\textit{Main Line Coloring:}
Let $V_j$ be the set of vertices among $\lbrace u,v_{1},\dots v_{r}\rbrace$ with 
color $j$, where at least two vertices have color $j$, i.e. $|V_{j}| > 1$ 
(vertices with a color that appears only once are clearly fixed).
	
We can assume that for all $ v \in V_j$, and for all $j$, there exists a 
leaf $w$ in $T_{v}$ of distance $d(v,w) \geq r(c,k)$.
Let $V_{j} = \{v_{1}^j, v_{2}^j, \dots, v_{m_j}^j\}$ for each of the colors 
$j = 1, \dots, \ell \leq c$ and do the following:
	
Consider (one of) the longest path(s) $R_i$ from $v_{i}^j$ to a leaf in 
$T_{v_{i}^j}$ for each of these vertices $v_{i}^j$, $i = 1, ..., m_j$.
We call the chosen paths $R_i$ \textit{main lines}.
To distinguish $\lbrace v_{1}^j, v_{2}^j, \dots, v_{m_j}^j \rbrace$ we color the 
main lines $R_i$ with pairwise different finite sequences of colorings.
	
For example, if $c = 2$, we color the paths with different 
``reverse binary colorings''. That means $R_1$ will be colored with $00000...$,
$R_2$ with $10000...$, $R_3$ with $01000...$, $R_4$ with $11000...$ and so on. 
If $c = 3$, we use ``reverse ternary colorings'', meaning that $R_1$ will be 
colored with $00000...$, $R_2$ with $10000...$, $R_3$ with $20000...$, $R_4$ 
with $01000...$ and so on, see Figure~\ref{fig:alg_coloring_ex1}.
	
\begin{figure}[h]
	\centering
	\includegraphics[scale=0.8]{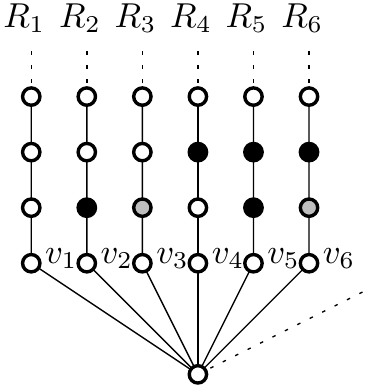}
	\caption{Coloring of main lines for $c = 3$.} \label{fig:alg_coloring_ex1}
\end{figure}
	
In general we use a reversed coloring based on the number system with base $c$.
Having done this for all $j \in \lbrace 1, \dots, \ell\rbrace$ continue with 
Step~4.

\textit{Step~4:} We consider all vertices $\lbrace w_{1}, \dots, w_{r} \rbrace$ 
produced in Step~3 that are part of one of the main lines and have valence 
$\geq 3$. That means we consider vertices that have at least one second 
up-neighbor, which is not in the main line.
Consider a vertex $w_{i}$. For simplicity we denote its up-neighbors 
(it has at least two) by $v_1$, $v_2$, $\dotsc$, $v_\ell$.
By construction only one vertex is in a main line, let it be $v_1$. Therefore, 
$v_1$ is already colored, whereas $v_2$, $\dotsc$, $v_\ell$ are uncolored. 
Let $a$ be the color of $v_1$.
We consider two cases:
	
\leftskip=1.5em
\textit{Case 1:} $\ell = 2$.
\leftskip=0in
	
Then color $v_2$ with $(a+1) \mod c$, see Figure~\ref{fig:thm_fix_tree_ex2} for 
an example, where $c = 3$.
	
\begin{figure}[h]
	\centering
	\includegraphics[scale=0.8]{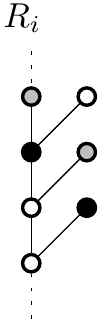}
	\caption{How to color unique siblings of vertices in a main line for 
	$c = 3$.}
	\label{fig:thm_fix_tree_ex2}
\end{figure}
	
If $c > 2$ continue with Step~4 for the next $w_i$. Otherwise do the following:

If the subtree $T_{v_2}$ contains only vertices of valence one or two (we say if 
there is no branching), color all vertices of $T_{v_2}-v_2$ with $1$.
Now, consider the case in which there is a branching in $T_{v_2}$. Color all 
vertices up to the branching in $T_{v_2}-v_2$ with $1$.
If there are~$2$ or~$3$ sibling vertices in that branching, color them all 
with~$1$. If the branching consists of four or more vertices, then assign them 
an optimal coloring, see Figure~\ref{fig:alg_coloring_ex2}. Restart Step~4 for 
the next $w_i$.
	
\begin{figure}[h]
	\centering
	\includegraphics[scale=0.8]{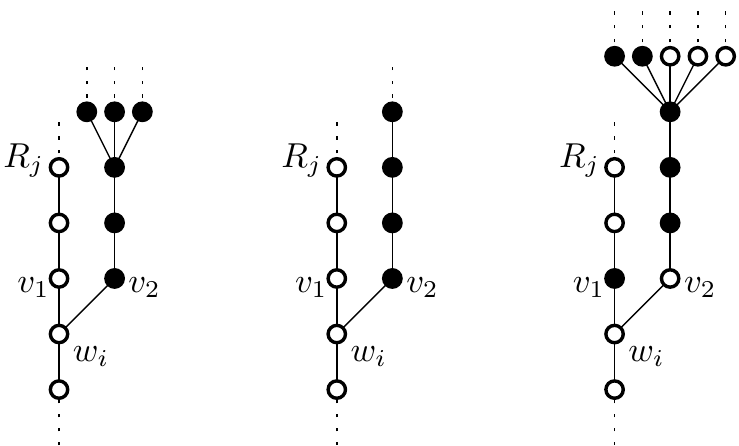}
	\caption{Coloring of secondary lines for $c = 2$.} 
	\label{fig:alg_coloring_ex2}
\end{figure}
	
\leftskip=1.5em
\textit{Case 2:} $3 \leq \ell \leq k - 1$.
\leftskip=0in
	
Color $v_2, v_3, \dotsc, v_\ell$ with an optimal coloring considering two 
additional restrictions:
	
First, we do not use the color $a$.
Second, if there exists $j \in \lbrace 2, \dots, \ell \rbrace$ such that $v_j$ 
has the color $b \in \lbrace 0, \dots, c-1 \rbrace \setminus \lbrace a \rbrace$, 
then there exists $j' \in \lbrace 2, \dots, \ell \rbrace$, $j \neq j'$,
such that $v_{j'}$ also has color $b$. This means that a color is always used at 
least twice, see Figure~\ref{fig:alg_coloring_ex3}.
Note  that $v_1$ is the only vertex among the siblings whose color appears just 
once.

Having done this for all $w_{1}, ..., w_{k}$ check whether there are 
indistinguishable vertices of the same color satisfying the distance condition
(as an example consider the case that there is a branching in  the subtree 
$T_{v_2}$ when $c = 2$ or within $\{v_2, v_3, \dotsc, v_\ell\}$), then repeat 
the Main Line Coloring method and apply Step~4 to these vertices. If there are 
no colored indistinguishable vertices, continue with Step~1.
 	
\begin{figure}[h]
	\centering
	\includegraphics[scale=0.8]{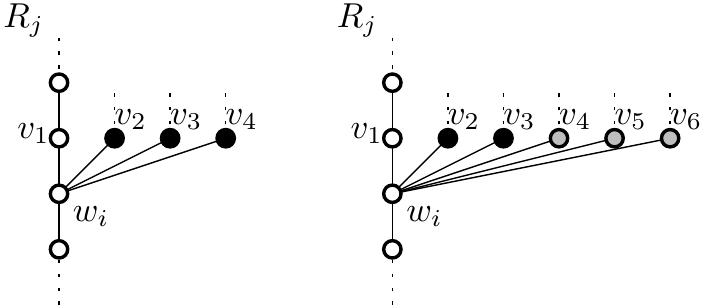}
	\caption{Examples of how to color the four (left) or the six (right) 
	up-neighbors of a vertex in some main line, with $c \geq 3$.} 
	\label{fig:alg_coloring_ex3}
\end{figure}
	
This completes the algorithm for the case that the center of $T$ is a 
single vertex.
	
Now, assume that the center of \(T\) is not a vertex but a an edge \(uv\). Let 
$T_u$ be the tree containing $u$ in $T - v$ and $T_v$ be the tree containing $v$ 
in $T - u$. Color \(u\) with $1$, \(v\) with $0$, and proceed with the coloring 
of $T_u$ and $T_v$ as explained above for trees whose center consists of a 
single vertex.

\subsection{Proof of the Main Theorem}

\begin{lemma} \label{lem:combi}
	Let $V$ be a set of $t$ vertices. Assume they are colored with an optimal 
	coloring consisting of $j$ colors, with the additional restriction
	that every color appears at least twice. Then, the maximal number $p$ of 
	vertices with the same color in $V$ is
	$\max\left\{3, \left\lceil\frac{t}{j}\right\rceil\right\}$.
\end{lemma}

\begin{proof}
	If $t$ is even and $j \geq \frac{t}{2}$, then there are enough colors such 
	that the vertices are colored pairwise differently, meaning that $p = 2$.
	If $t$ is even but $j < \frac{t}{2}$, we are forced to use every possible 
	color, and every color has to be used at least twice.
	Thus, the restriction is automatically fulfilled and 
	$p = \left\lceil\frac{t}{j}\right\rceil$.
	
	Now, consider the case when $t$ is odd. We first ignore one vertex and 
	proceed as in the even case for the remaining $t - 1$ vertices.
	The ignored vertex then has to get the same color as one of the remaining 
	vertices due to the additional restriction.
	That means
	\begin{align*}
		\text{if } j \geq \frac{t-1}{2}, 
			\text{ then } p &= 2 + 1 = 3, \text{ and }\\
		\text{if } j < \frac{t-1}{2}, 
			\text{ then } p &= \left\lceil\frac{t-1}{j}\right\rceil + 1.
	\end{align*}
	The $+1$ arises from the vertex that was first ignored. It is needed in the 
	case where each color is used equally often.
	It is easy to verify that 
	$\left\lceil\frac{t}{j}\right\rceil \geq 
	\left\lceil\frac{t-1}{j}\right\rceil + 1$ 
	for every possible pair of $t$ and $j$.
	Therefore an upper bound for $p$ is always 
	$\max\left\{3, \left\lceil\frac{t}{j}\right\rceil\right\}$.
\end{proof}

\begin{lemma} \label{lem:max}
	Let $T$ be a finite tree with maximal valence $0 < k < \infty$. 
	If $T$ is colored as described in the Coloring Algorithm
	where $c \geq 2$, then the largest number of sibling vertices fulfilling the 
	distance condition and having the same color that can appear is
	$\max\left\{3, \left\lceil\frac{k-2}{c-1}\right\rceil\right\}$.
\end{lemma}

\begin{proof}
	Within our Coloring Algorithm there are three situations in which sibling 
	vertices with the same color fulfilling the distance condition might appear.

	\textit{Situation~1:} If we apply an optimal coloring to $m$ vertices, 
	we obtain a maximum of $\left\lceil\frac{m}{c}\right\rceil$
	vertices with the same color. For a tree with maximal valence $k$ it is 
	bounded by $s_1 = \left\lceil\frac{k}{c}\right\rceil$.

	\textit{Situation~2:} In Step 4 case 2 of the Coloring Algorithm, we use an 
	optimal coloring with $c-1$ colors such that each
	color appears at least twice. By applying Lemma~\ref{lem:combi}, we obtain a 
	maximum number
	$s_2 = \max\left\{3, \left\lceil\frac{k-2}{c-1}\right\rceil\right\}$ of 
	sibling vertices with the same color in that case.

	\textit{Situation~3:} Consider the case $c = 2$ in Step 4 Case 1 of the 
	algorithm where we are in the branching situation. There are at most 
	$s_3 = \max\{3,\lceil\frac{k-1}{2}\rceil\}$ vertices with the same color. 
	We see that $s_2 \geq s_3$ for $c = 2$ and $k \geq 4$.

	It remains to compute the maximum of $s_1$ and $s_2$.
	Straight forward calculations show that
	\begin{align*}
		k < 2c 
		& \quad \Longleftrightarrow \quad \frac{k-2}{c-1} < \frac{k}{c} < 2.
	\end{align*}
	Therefore, for $k < 2c$ the maximum is $3$, while for $k \geq 2c$ it is 
	$\frac{k-2}{c-1}$.
	Thus, the maximum of $s_1$ and $s_2$ is 
	$\max\left\{3, \left\lceil\frac{k-2}{c-1}\right\rceil\right\}$.
\end{proof}

We are now able to prove Proposition~\ref{thm:fix_tree}.

\begin{proof}[Proposition \ref{thm:fix_tree}]
	The cases $k \in \{0, 1, 2\}$ and $c \geq k$ are trivial.
	For $c = k-1$, we refer to Lemma~\ref{co:k-1_coloring} in 
	Section~\ref{sec:finite_trees}.
	So, assume that we have $2 \leq c \leq k-2$ and $k \geq 4$.
	
	\smallskip
	
	First, we consider a finite tree $T$ and apply the Coloring Algorithm.
	Assume that the center of $T$ consists of a single vertex $v$. This vertex 
	is fixed by each automorphism of $T$.
	Thus, it remains to show that for all \(n\in \mathbb{N}\) the vertices with 
	the same color in $S(n,v)$  which satisfy the distance condition
	are indistinguishable due to our algorithm.

	The distance $r(c,k)$ is built upon the maximal number of  indistinguishable 
	vertices with the same color that can appear in the same sphere and
	the given Main Line Coloring in the algorithm. The aim of these main lines 
	is to fix indistinguishable siblings that have  the same color and
	the vertices on these main lines themselves.

	Let $n' \in \mathbb{N}$ be the smallest index such that there exist 
	indistinguishable sibling vertices with the same color
	$v_1,\dots, v_m$, $m \geq 2$, in $S(n',v)$ which fulfill the distance 
	condition.

	Consider $v_1,\dotsc, v_m$ and their main lines $R_i$ given by Step~3 of the 
	algorithm. We see that if each $v_i$ has distance at least
	$\log_c m$ to a leaf in $T_{v_i}$, then the main lines \(R_i\) are colored 
	pairwise differently. Due to Lemma~\ref{lem:max} this distance is
	bounded by 
	$ \log_c \max\left\{3, \left\lceil\frac{k-2}{c-1}\right\rceil\right\}$. 
	Since this coincides with our distance condition regarding
	the given $r(c,k)$, for any color-preserving automorphism $\alpha$ of $T$  
	with $\alpha(v_i)=v_j$ and $\alpha(R_i)=R_j$ for some
	$i,j\in \{ 1,\dotsc, m\}$, we see that $i=j$. This means that the main lines 
	are not interchangeable.
	
	Next we argue why a main line $R_i$ cannot be mapped to \textit{any} other 
	string of some $T_{v_j}$. By a string we mean a path, that has at most one 
	vertex of each sphere. Therefore we show that there is no
	automorphism $\alpha$ such that $\alpha(T_{v_i})=T_{v_j}$ for any 
	$i,j\in\{1,\dotsc, m\}$.

	Without loss of generality we consider $T_{v_i}$ and $T_{v_j}$, $i \neq j$ 
	and assume that there exists at least one vertex in $R_i$ resp. $R_j$
	with at least two up-neighbors (otherwise $T_{v_i}$ and $T_{v_j}$ are 
	stabilized by the Main Line Coloring method).

	\textit{Case 1:} There exists a vertex $w_0$ in $R_i$ with one up-neighbor 
	$w_1$ in $R_i$ and at least two more up-neighbors, called
	$w_2, \dots , w_\ell$, $\ell \geq 3$.
	Let us assume that there exists a color-preserving automorphism $\alpha$ of 
	$T$ that maps $T_{v_j}$ to $T_{v_i}$. By the Main Line Coloring method
	we know that $\alpha(R_i) \neq \alpha(R_j)$. Thus, there exists a vertex 
	$\tilde{w} \in R_j$ (in the same sphere as $w_1, \dotsc, w_{\ell}$)
	and $k \in \lbrace 2, \dots ,\ell \rbrace$ such that 
	$\alpha(\tilde{w}) = w_{k}$. See Figure~\ref{fig:thm_fix_tree_ex1}.

	\begin{figure}[h]
		\centering
		\includegraphics[scale=0.8]{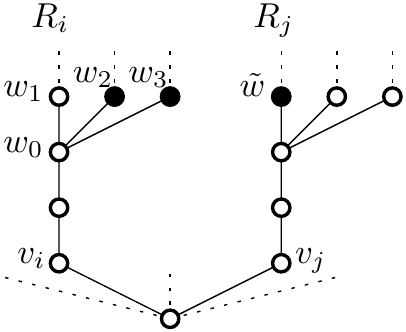}
		\caption{Example of $w_0$ in $R_i$ with three up-neighbors. 
		The vertex $\tilde{w}$ has the same color as $w_2$ and $w_3$.} 
		\label{fig:thm_fix_tree_ex1}
	\end{figure}

	Due to Step~4, Case~2, in the algorithm, $\tilde{w}$ does not have a sibling 
	vertex with the same color, whereas $w_k$ for sure has.
	So, $\alpha$ cannot map $\tilde{w}$ to $w_k$, and we have a contradiction. 
	We conclude that $\alpha$ does not exists.

	\textit{Case 2:} There exists a vertex $w_0$ in $R_i$ with one up-neighbor 
	$w_1$ in $R_j$ and exactly one additional up-neighbor, called $w_2$.

	Let us again assume that there exists a color-preserving automorphism 
	$\alpha$ of $T$ that maps $T_{v_j}$ to $T_{v_i}$. By main line
	coloring we know that $\alpha(w_1) \notin R_j$. Thus there exits a vertex 
	$\tilde{w} \in R_j$ such that $\alpha(w_{2}) = \tilde{w}$.

	\textit{Case 2.1:} Let $c \geq 3$.
	If $\alpha$ swaps $\tilde{w}$ and $w_{2}$, then $\alpha$ also swaps $w_1$ 
	and the unique sibling of $\tilde{w}$. But if $\tilde{w}$ and
	$w_{2}$ have the same color, their siblings do not, because of the  shifting 
	of the colors modulo $c$ in Case~1 of Step~4 of the algorithm.
	Thus, such an $\alpha$ does not exist.
	Note,  here it is important that we assume that $c \geq 3$. For $c = 2$, we 
	only have the color pairs $(0,1)$ and $(1,0)$ which cannot be distinguished.

	\textit{Case 2.2:} Let $c = 2$.
	If the subtree $T_{w_2}$ contains only vertices of valence one or two (no 
	branching), all vertices of $T_{w_2}$ (except maybe $w_2$) are
	colored with $1$. In contrast to this, at least the last vertex of $R_j$ 
	(the leaf) is colored with~0. This is due to the given $r(2,k)$,
	see Figure \ref{fig:alg_coloring_ex3} (here we need the +1 in the case where 
	$r(c,k) = \log_2(\max\{3, k - 2\}) + 1$). 
	Thus, such an $\alpha$ does not exist.
	
	\begin{figure}[h]
		\centering
		\includegraphics[scale=0.8]{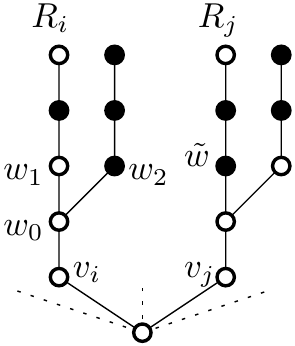}
		\caption{The distance $r(2,k)$ guarantees that the last vertex of a main 
		line is always white if $c = 2$.} \label{fig:thm_fix_tree_ex3}
	\end{figure}

	Now, assume there is a branching in $T_{w_2}$. By our algorithm, we have 
	avoided that there exist single vertices of a
	certain color in such a branching (see the left and the right picture in 
	Figure~\ref{fig:alg_coloring_ex2}) while in every branching in the
	main line of $R_j$ there exist a vertex (the vertex in the main line itself) 
	with color $d$ that is the only vertex with this color in that
	branching\footnote{Coloring the neighbors of a  vertex in $R_j$ we excluded 
	the use of the color of this particular vertex.}.
	
	All in all, we showed that  $v_{1}, \dots, v_{m}$ are fixed.
	Now, we can use a final inductive argument. Assume that all vertices with 
	the same color up to the sphere $S(n,v)$ are fixed.
	Consider indistinguishable sibling vertices that fulfill the distance 
	condition in the sphere $S(n+1,v)$. At this point,
	we can apply the same argument as above to ensure that they have to be 
	fixed. Note that here it is important to assume that
	everything below these vertices is already fixed.

	If the center of $T$ consists of an edge \(uv\), color $u$ with $0$ and $v$ 
	with $1$. Then \(u\) and \(v\) are fixed by the properties of a
	tree and we can apply the Coloring Algorithm to the remaining vertices in 
	the two subtrees containing $u$ and $v$ after removing the
	edge $uv$ rooted in, respectively, $u$ and $v$. Then, we use the same 
	reasoning  as above. This completes the proof for the finite case.

	Now, let $T$ be an infinite, locally finite tree.
	We first assume that $T$ has at least two ends. In that case consider the 
	trunk $T^C$ of $T$ which is a locally finite, infinite tree without leaves.
	Since it is unique any automorphism of $T$ leaves $T^C$ invariant.
	By~\cite{wazh-2007} we know that $T^C$ is $2$-distinguishable.
	Now, for every vertex $u$ in the trunk we consider the subtree $T_u$ that 
	contains, as explained above, the vertex $u$ together with all
	finite subtrees of $T-u$.
	We apply the Coloring Algorithm to each of these subtrees and fix all 
	vertices which fulfill the distance condition.
	
	If $T$ has only one end, then there exists a vertex $v$ with valence $1$. 
	Let $R$ be the ray with root $v$. We color all vertices of $R$
	with $0$. For each vertex in $R$ we color the maximal $k-2$ neighbors not in 
	$R$ with an optimal coloring with the $c-1$ colors different
	from $0$. Thus, there are at most $\lceil \frac{k-2}{c-1}\rceil$ vertices 
	with the same color.
	If there are indistinguishable vertices we apply the Coloring Algorithm from 
	above to fix all vertices which fulfill the distance condition.
\end{proof}

\begin{proof}[Main Theorem]
	If $c = k$, we have enough colors to fix all vertices of the tree, 
	so $r = 0$. If $c = k - 1$, we refer to Theorem~\ref{thm:inftree} and 
	Lemma~\ref{co:k-1_coloring} in Section~\ref{sec:finite_trees}.
	Assume now that $c = 2$ and that we have an integer $r$ satisfying 
	$3 < k \leq 2^{r-1}$. Then $\log_2k \leq r - 1$. Since $3 < k$, we conclude 
	that $\log_{2}(\max \left\{3, k-2\right\}) + 1<r$. The claim of the Theorem 
	follows then by Proposition~\ref{thm:fix_tree}.
	Consider now the remaining case that $c > 2$ and $r$ is an integer such that 
	$3 < k\leq c^r(c-1)+2$. Then we have $\frac{k-2}{c-1}\leq c^r$. Since $c^r$ 
	is an integer, it follows that $\log_c \lceil\frac{k-2}{c-1}\rceil \leq r$. 
	If $\lceil \frac{k-2}{c-1}\rceil \geq 3$ the claim follows by 
	Proposition~\ref{thm:fix_tree}. Furthermore we have 
	$\log_c \frac{1}{c-1} < r$, since $3 < c^r(c-1) + 2$, which implies $1 < r$. 
	Thus $\log_c 3 < r$ and the claim follows again by 
	Proposition~\ref{thm:fix_tree}.
\end{proof}

We end with an example that shows that the given constant $r(c,k)$ in 
(\ref{ali:definition_r(k)}) is tight in some special cases.

\begin{example}\label{ex:max_val_7}
	Consider a tree as in Figure \ref{fig:ex_tight} with maximal valence 
	$k = 10$, which we would like to color with $c = 3$ colors. Without a  
	coloring the vertices in the first sphere  are indistinguishable.
	Using an optimal coloring for these vertices there are 
	$\lceil \frac{10}{3} \rceil = 4$ vertices $v_1, \dotsc, v_4$ with 
	the same color.
	Thus in our algorithm they are starting points of main lines of length two 
	and we can distinguish them in that way, see Figure~\ref{fig:ex_tight}. 
	Clearly, it is not possible to distinguish $v_1, \dotsc, v_4$ by using only  
	$3$-colors in the next sphere, but we can fix them by coloring all vertices 
	of the next two spheres. That is what  the upper bound
	\(r(3,10) =  \log_{3}\max\{3,\frac{10-2}{2}\}\approx 1.26\) yields, which 
	means we need at least distance~$2$.
	
	\begin{figure}[h]
		\centering
		\includegraphics[scale=0.8]{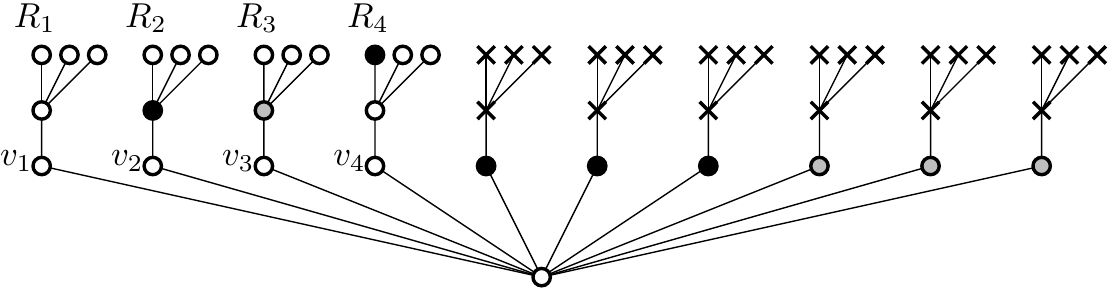}
		\caption{Example of a $3$-colored tree with maximal valence $k = 10$.} 
		\label{fig:ex_tight}
	\end{figure}
\end{example}

One easily deduces that the number of vertices which can be distinguished by the 
given coloring in Proposition~\ref{thm:fix_tree} depends on the structure
of the graph. Especially it depends on the number and the distribution of the 
leaves. Of course there are enough cases where it is possible to
distinguish considerably more vertices through a coloring than provided by our 
theorem (e.g. if a lot of vertices are fixed by the structure of the tree).
But for a general result one has to consider the critical cases as in 
Example~\ref{ex:max_val_7} where we have seen that the bound is tight.

We always considered the center of a finite tree as a fixed starting vertex in 
our Coloring Algorithm. But one also could take another vertex of the graph, 
that is fixed by all automorphisms. In some cases this may lead to colorings, 
that fix more vertices of the graph than in the case where we start the coloring 
algorithm in the center.

\begin{acknowledgement}
	The authors acknowledges the support of the Austrian Science Fund (FWF) 
	W1230 and of the Award 317689 of the Simons Foundation.
\end{acknowledgement}

\end{document}